\documentclass[a4,11pt]{amsart}
\usepackage{amssymb,latexsym,color}
\usepackage{amsmath,amssymb}
\usepackage[all]{xy}

\def\set#1{ \{#1\}}

\newcommand\free{{\mathfrak F}}

\newcommand\ide{{1'}}
\newcommand\comp{\mathbin{;}}
\newcommand\meet{\mathbin{\cdot\,}}

\newcommand\join{\mathbin{+}}

\newcommand\path{{\rm path}}

\newcommand\Lang{{\sf L}}
\newcommand\Rel{{\sf R}}
\newcommand\R{{\sf R}}
\newcommand\V{{\sf V}}

\newcommand\K{{\sf K}}

\newtheorem{theorem}{Theorem}[section]
\newtheorem{lemma}[theorem]{Lemma}
\newtheorem{corollary}[theorem]{Corollary}

\newtheorem{rem}[theorem]{Remark}
\newtheorem{problem}[theorem]{Problem}

\newtheorem{defin}[theorem]{Definition}

\newenvironment{definition}{\begin{defin}\rm}{\end{defin}}
\newenvironment{remark}{\begin{rem}\rm}{\end{rem}}

\title[Ordered Monoids]
{Ordered Monoids: Languages and Relations}

\author{Szabolcs Mikul\'as}
\address{Department of Computer Science and Information Systems\\
Birkbeck, University of London\\
szabolcs@dcs.bbk.ac.uk
}
\date{}

\begin{document}

\maketitle

\begin{abstract}
We give a finite axiomatization for
the variety generated by 
relational, integral ordered monoids.
As a corollary we get a finite axiomatization for
the language interpretation as well.
\\
{\sc Keywords:}
finite axiomatizability, algebras of relations, language algebras
\end{abstract}

\section{Introduction}

We will focus on algebras of similarity type
$\Lambda\subseteq (\join,\meet,\comp,0,\ide)$
where $\join,\meet,\comp$ are binary operations and $0,\ide$ are constants.
We will consider two types of representations:
as families of binary relations and as families of languages.
Our main concern is whether the varieties generated by these
interpretations are finitely axiomatizable.

\begin{definition}[Language algebras]
Let $\mathfrak{A}=(A,\Lambda)$ be an algebra of similarity type 
$\Lambda\subseteq (\join,\meet,\comp,0,\ide)$.
We say that $\mathfrak{A}$ is a \emph{language algebra} if
the following holds.
There is an alphabet $\Sigma$ such that
$A$ is a family of languages over $\Sigma$,
i.e., $A\subseteq\wp(\Sigma^*)$ where $\Sigma^*$ denotes the
set of words (finite strings) over $\Sigma$,
and the operations in $\Lambda$ are interpreted as follows:
join $\join$ is union, meet $\meet$ is intersection,
composition $\comp$ is concatenation
\begin{equation*}
a\comp b=\{ st: s\in a\mbox{ and } t\in b\}
\end{equation*}
$0$ is the empty language $\emptyset$ and
$\ide$ is the singleton language 
consisting of the empty word. 
\end{definition}

We will denote the class of language algebras of similarity type $\Lambda$ by
$\Lang(\Lambda)$.

\begin{definition}[Relation algebras]
Let $\mathfrak{A}=(A,\Lambda)$ be an algebra of similarity type 
$\Lambda\subseteq (\join,\meet,\comp,0,\ide)$.
We say that $\mathfrak{A}$ is a \emph{relation algebra} if
the following holds.
There is a set $U$, the \emph{base} of $\mathfrak{A}$, such that
$A$ is a family of binary relations on $U$,
i.e., $A\subseteq \wp(U\times U)$, 
and the operations in $\Lambda$ are interpreted as follows:
join $\join$ is union, 
meet $\cdot$ is intersection,
composition $\comp$ is relation composition
\begin{align*}
a\comp b&=\{(u,v)\in U\times U: 
(u,w)\in a \mbox{ and } (w,v)\in b \mbox{ for some } w\}\\
\intertext{$\ide$ is the identity relation on $U$}
\ide&=\{(u,v)\in U\times U: u=v\}
\end{align*}
and $0$ is the empty relation $\emptyset$.
\end{definition}

We will denote the class of relation algebras of similarity type $\Lambda$ by
$\Rel(\Lambda)$.

In passing we note that the representation classes
$\Lang(\Lambda)$ and $\Rel(\Lambda)$ are not finitely axiomatizable
whenever $(\meet,\comp,\ide)\subseteq \Lambda$.
Indeed, the class of language algebras is not closed under products (see below),
while the quasivariety $\Rel(\meet,\comp,\ide)$ has no finite base
\cite{HM-rep-07}.
This is one of the main reasons why we concentrate on the generated varieties below.

Assume that $0,\ide\in \Lambda$ and let $\mathfrak{A}$ be a $\Lambda$-algebra.
We say that $\mathfrak{A}$ is \emph{integral} if $\ide$ is an atom of $\mathfrak{A}$,
i.e., $\ide$ is a minimal non-zero element.
For a class $\K(\Lambda)$ of $\Lambda$-algebras,
let $\K^i(\Lambda)$ denote the subclass of integral $\Lambda$-algebras.
Observe that every language algebra is integral, $\Lang(\Lambda)=\Lang^i(\Lambda)$, 
while there are non-integral relation algebras, $\Rel(\Lambda)\supset\Rel^i(\Lambda)$.

For a class $\K(\Lambda)$ of $\Lambda$-algebras,
let $\V(\K(\Lambda))$ denote the variety generated by $\K(\Lambda)$.
Note that the variety $\V(\K^i(\Lambda))$ generated by integral algebras
may contain non-integral algebras,
since neither products nor homomorphisms preserve the property that $\ide$ is an atom.

\section{Main results}

First we look at 
$\Lambda=(\meet,\comp,0,\ide)$.
As usual $x\le y$ is defined by $x\meet y=x$ and we assume that
$\comp$ binds closer than $\meet$, e.g., we write $x\meet y\comp z$
for $x\meet(y\comp z)$.

We define $\mbox{\rm Ax}(\meet,\comp,0,\ide)$ 
as the collection of the following axioms.
\begin{itemize}
\item[]
Semilattice axioms (for $\meet$)
\item[]
Monoid axioms (for $\comp$ and $\ide$)
\item[]
Monotonicity:
\begin{equation}\label{eq:mon}
(x\meet x')\comp (y\meet y')\le x\comp y
\end{equation}
\item[]
Axioms for $0$:
\begin{align}
0&=0\meet x\\
0&=0\comp x= x\comp 0
\end{align}
\item[]
Subidentity axioms:
\begin{align}
(\ide\meet x)\comp(\ide\meet y)&=\ide\meet x\meet y \label{eq:fun}\\
(\ide\meet x)\comp(y\meet z)&=(\ide\meet x)\comp y \meet z \label{eq:fun1l}\\
(x\meet y)\comp(\ide\meet z)&=x \meet y\comp(\ide\meet z) \label{eq:fun1r}
\end{align}
\end{itemize}
It is easily checked that all these axioms are valid in both relation and language algebras.

We will need additional axioms that are valid in language algebras
and in integral relation algebras.
We define $\mbox{\rm Ax}^i(\meet,\comp,0,\ide)$ as
$\mbox{\rm Ax}(\meet,\comp,0,\ide)$ augmented with the integral axioms~\eqref{eq:int1}
and~\eqref{eq:int2} below.
\begin{itemize}
\item[]
``Integrality'':
\begin{align}
\ide\meet x\comp y &= \ide\meet y\comp x \label{eq:int1} \\
(\ide\meet x)\comp y&= y\comp(\ide\meet x) \label{eq:int2}
\end{align}
\end{itemize}
Note that $\ide\meet x\comp y=0$ iff $\ide\meet y\comp x=0$
in a relation algebra $\mathfrak{A}$.
Hence, if $\mathfrak{A}$ is integral, then it follows that the two terms in~\eqref{eq:int1}
are either $0$ or $\ide$ at the same time.
Checking validity of the other integral axiom in integral algebras is similar.
The set $\mbox{\rm Ax}^i(\meet,\comp,0,\ide)$ is not independent,
e.g., we can derive~\eqref{eq:fun1r} from~\eqref{eq:fun1l} with the use of~\eqref{eq:int2}.

\subsection{Axiomatizations for integral relation algebras}


\begin{theorem}\label{thm:main}
The variety $\V(\Rel^i(\meet,\comp,0,\ide))$ generated by 
$\Rel^i(\meet,\comp,0,\ide)$
is axiomatized by $\mbox{\rm Ax}^i(\meet,\comp,0,\ide)$.
\end{theorem}

\begin{proof}
By the validity of the axioms in integral relation algebras
we get that equations derivable using equational logic from $\mbox{\rm Ax}^i(\meet,\comp,0,\ide)$ 
are valid in $\V(\Rel^i(\meet,\comp,0,\ide))$.
We will prove the other direction by showing that, for every non-derivable equation, 
there is an algebra in~$\Rel^i(\meet,\comp,0,\ide)$ 
witnessing that the equation is not valid,
see Lemma~\ref{lem:wit}.
\end{proof}

Next we include union into the signature.
Define $\mbox{\rm Ax}^i(\join,\meet,\comp,0,\ide)$
as $\mbox{\rm Ax}^i(\meet,\comp,0,\ide)$
augmented with the distributive lattice axioms (for $\meet$ and $\join$)
and that the operation $\comp$ is additive:
\begin{align}
(x\join y)\comp z&= x\comp z \join y\comp z\\
x\comp (y\join z)&= x\comp y \join x\comp z
\end{align}
for every $x,y,z$.

\begin{theorem}\label{thm:dl}
The variety $\V(\Rel^i(\join,\meet,\comp,0,\ide))$ generated by 
$\Rel^i(\join,\meet,\comp,0,\ide)$
is axiomatized by $\mbox{\rm Ax}^i(\join,\meet,\comp,0,\ide)$.
\end{theorem}

\begin{proof}
This follows from \cite[Corollary~2]{bre:93}
(see also \cite[Proposition~4.4]{AM-axi-11}), we just give a sketch here.

By the validity of the axioms we get that derivable equations are valid.
For the other direction assume that $\V(\Rel^i(\join,\meet,\comp,0,\ide))\models a = b$.
Thus we have both $\V(\Rel^i(\join,\meet,\comp,0,\ide))\models a\le b$
and $\V(\Rel^i(\join,\meet,\comp,0,\ide))\models b\le a$, where
$x \le y$ is defined by $x\meet y=x$.
Note that, using the additivity of the operations in relation algebras, 
we can rewrite every term as a join of join-free terms. 
Thus $a\le b$ can be equivalently rewritten as 
$a_1\join\dots\join a_n \le b_1\join\dots \join b_m$
where $a_i$ and $b_j$ do not contain $\join$.
Furthermore, using the term graphs of \cite{AB-equ-95} one can show that
$a_1\join\dots\join a_n\le b_1\join\dots \join b_m$ is valid in 
$\V(\Rel^i(\join,\meet,\comp,0,\ide))$
iff, for every $i$, there is $j$ such that $a_i\le b_j$ is valid in
$\V(\Rel^i(\join,\meet,\comp,0,\ide))$.
Thus $\V(\Rel^i(\meet,\comp,0,\ide))\models a_i\le b_j$
(since $a_i$ and $b_j$ do not contain $\join$).
By Theorem~\ref{thm:main} we get
$\mbox{\rm Ax}^i(\meet,\comp,0,\ide)\vdash a_i\le b_j$, whence
$\mbox{\rm Ax}^i(\join,\meet,\comp,0,\ide) \vdash a_i\le b_j$.
Since the distributive lattice axioms ensure that the ordering
$x\le y$ can be equivalently defined by either $x\meet y=x$ or $x\join y=y$,
we get $\mbox{\rm Ax}^i(\join,\meet,\comp,0,\ide) \vdash a_i\le b_1\join\dots \join b_m$.
Since this holds for every $i$, using the additivity axioms we get
$\mbox{\rm Ax}^i(\join,\meet,\comp,0,\ide) \vdash a\le b$.
By an identical argument we get
$\mbox{\rm Ax}^i(\join,\meet,\comp,0,\ide) \vdash b\le a$ as well, 
whence $\mbox{\rm Ax}^i(\join,\meet,\comp,0,\ide) \vdash a= b$,
as desired.
\end{proof}

\subsection{Axiomatization for language algebras}

\begin{theorem}\label{thm:lang}
The variety $\V(\Lang(\join,\meet,\comp,0,\ide))$ generated by 
$\Lang(\join,\meet,\comp,0,\ide)$
is axiomatized by 
\begin{equation}\label{eq:empty}
x\comp y\meet \ide = (x\meet\ide)\comp(y\meet\ide)
\end{equation}
together with
$\mbox{\rm Ax}^i(\join,\meet,\comp,0,\ide)$.
\end{theorem}

\begin{proof}
In \cite{AMN-equ-11} it is shown that the equational theory
of $\Lang(\join,\meet,\comp,0,\ide)$ is finitely axiomatizable
over the equational theory of $\Rel(\join,\meet,\comp,0,\ide)$
by 
the equations~\eqref{eq:empty} and~\eqref{eq:int2}.
By a minimal and straightforward modification of the proof of
\cite[Corollary~3.7]{AMN-equ-11} we get that the equational theory
of $\Lang(\join,\meet,\comp,0,\ide)$ is finitely axiomatizable
by~\eqref{eq:empty}
over the equational theory of $\Rel^i(\join,\meet,\comp,0,\ide)$,
which is axiomatized by $\mbox{\rm Ax}^i(\join,\meet,\comp,0,\ide)$.
\end{proof}

\section{Relational representation of the free algebra}

In this section we make the proof of Theorem~\ref{thm:main} complete.
We will need the following easy consequences of
$\mbox{\rm Ax}(\join,\meet,\comp,0,\ide)$:
\begin{equation}
(e_1\comp x\comp e_2)\meet(e_3\comp x \comp e_4)=
(e_1\meet e_3)\comp x\comp (e_2\meet e_4)
\label{eq:con1}
\end{equation}
\begin{equation}
e\comp(x\meet y)=e\comp x\meet e\comp y= e\comp x\meet y= x\meet e\comp y
\label{eq:con2}
\end{equation}
where $e,e_1,\ldots , e_4$ are subidentity terms (have the form $z\meet\ide$).

\subsection{Step-by-step construction}

Fix a countable set $X$ of variables.
Let $T_X$ be the set of $(\meet,\comp,0,\ide)$-terms
using the variables from $X$ and 
${\mathfrak F}_X=(F_X,\meet,\comp,0,\ide)$ be the 
free algebra of the variety defined by $\mbox{\rm Ax}^i(\meet,\comp,0,\ide)$
which is freely generated by $X$.
That is, $\mathfrak{F}_X$ is given by factoring the absolutely free algebra
by the congruence of derivability from $\mbox{\rm Ax}^i(\meet,\comp,0,\ide)$
in equational logic.
Thus $\free_X\models\tau\le\sigma$
(under the natural valuation of evaluating every variable to itself)
iff
$\mbox{\rm Ax}^i(\meet,\comp,0,\ide)\vdash \tau\le\sigma$,
where $\vdash$ denotes derivability in equational logic.

By a filter ${\mathcal F}$ of ${\mathfrak F}_X$ we mean 
a subset closed upward and under meet $\meet$,
that is,  
$\tau,\sigma\in {\mathcal F}$ iff
$\tau\meet\sigma\in {\mathcal F}$.
For a $S\subseteq F_X$, let $\mathcal{F}(S)$ denote the filter generated by $S$.
In particular, for a term $\tau$, $\mathcal{F}(\{\tau\})$ denotes 
the principal filter generated by $\{\tau\}$, i.e., 
the upward closure 
$\{\tau\}^\uparrow$ of the singleton set $\{\tau\}$. 
We extend the operation $\comp$ to subsets of elements as follows:
\begin{align*}
X\comp Y &=\set{x\comp y:x\in X, y\in Y}
\end{align*}
for subsets $X,Y$.
When $X=\{ x\}$ is a singleton, we will also use the notation 
$x\comp Y=\{ x\comp y:y\in Y\}$, $\{ x\}^\uparrow= x^\uparrow$, 
$\mathcal{F}(\{\tau\})=\mathcal{F}(\tau)$, etc.

For the whole rest of the section we fix a term $\theta$ such that
$\mbox{\rm Ax}^i(\meet,\comp,0,\ide)\not\vdash\theta=0$.

Next we define special filters that are generated by some set of subidentity elements. 
We define 
\begin{align*}
\mathcal{E}(\tau):&=\mathcal{F}(\{\epsilon\le\ide:\epsilon\comp\tau\comp \epsilon=\tau\})&&\\
&=\mathcal{F}(\{\epsilon\le\ide:\tau\comp \epsilon=\tau\})&&\mbox{by~\eqref{eq:int2},\eqref{eq:fun}}\\
&=\mathcal{F}(\{\epsilon\le\ide:\epsilon\comp\tau=\tau\})&&\mbox{by~\eqref{eq:int2},\eqref{eq:fun}}
\end{align*}
for each element $\tau$.
It is worth noting that
\begin{equation}\label{eq:emon}
\tau\le\sigma\mbox{ implies }\mathcal{E}(\sigma)\subseteq\mathcal{E}(\tau)
\end{equation}
since $\epsilon\comp\tau=\epsilon\comp(\tau\meet\sigma)=\tau\meet\epsilon\comp\sigma\ge
\tau\meet\sigma=\tau$ by \eqref{eq:con2} whenever $\epsilon\in\mathcal{E}(\sigma)$.
We denote $\mathcal{E}:=\mathcal{E}(\theta)$ for our fixed term $\theta$.
We say that the filter $\mathcal{F}$ is \emph{fundamental}
if the following holds:
there is an element $\tau$ such that
$\mathcal{F}=\mathcal{F}(\mathcal{E}\comp\tau\comp\mathcal{E})$.
Observe that $\mathcal{E}=\mathcal{F}(\mathcal{E}\comp\ide\comp\mathcal{E})$
is fundamental, since for subidentity elements $\epsilon_1,\epsilon_2$, we have
$\epsilon_1\comp\ide\comp \epsilon_2=\epsilon_1\comp \epsilon_2=\epsilon_1\meet \epsilon_2$
by~\eqref{eq:fun}.
Also note that
$\mathcal{F}(\mathcal{E}\comp\tau\comp\mathcal{E})=
(\mathcal{E}\comp\tau\comp\mathcal{E})^\uparrow$,
since 
$(\epsilon_1\comp\tau\comp \epsilon_2)\meet(\epsilon_3\comp\tau\comp \epsilon_4)=
(\epsilon_1\meet \epsilon_3)\comp\tau\comp(\epsilon_2\meet \epsilon_4)$
by~\eqref{eq:con1}.

We will define a chain of labelled, directed graphs $G_n= (U_n,\ell_n,E_n,W_n)$
for $n\in\omega$, where 
\begin{itemize}
\item
$U_n$ is the set of nodes,
\item
$\ell_n\colon U_n\times U_n\to\wp(F_X)$ is a labelling of edges,
\item
$E_n:=\{(u,v)\in U_n\times U_n: \ell_n(u,v)\ne\emptyset\}$
is the set of edges with non-empty labels,
\item
$W_n\subseteq E_n$ is a distinguished set of \emph{witness edges}.
\end{itemize}
We will make sure that the following inductive conditions are maintained during the construction:
\begin{description}
\item[RT]
$E_n$ is reflexive and transitive.
\item[Gen]
$W_n$ generates $E_n$ by closing under transitivity.
\item[Fun]
for every $(u,v)\in E_n$, $\ell_n(u,v)$ is a proper, fundamental filter:
there is $\tau$ such that 
$\ell_n(u,v)=\mathcal{F}(\mathcal{E}\comp\tau\comp\mathcal{E})$.
\item[DR]
for every $(u,v)\in E_n$, 
$\mathcal{E}(\sigma)\subseteq\ell_n(u,u)=\ell_n(v,v)=\mathcal{E}$
for every $\sigma\in\ell_n(u,v)$.
\item[Comp]
for every $(u,v),(u,w),(w,v)\in E_n$, we have 
$\ell_n(u,w)\comp\ell_n(w,v)\subseteq\ell_n(u,v)$.
\item[Ide]
for every $(u,v)\in E_n$, if $\ide\in\ell_n(u,v)$, then $u=v$.
\end{description}
Observe that DR implies 
$\ide\in\ell_n(u,u)=\ell_n(v,v)$ for every $u$ and $v$.

These conditions, with the exception of Ide, will be easily seen
to be maintained during the construction. 
We will check that Ide holds for witness edges $(u,v)\in W_n$ as well,
but we will establish the general case for Ide only at the end of the section.

The construction will terminate in $\omega$ steps, 
yielding $G_\omega=(U_\omega,\ell_\omega,E_\omega,W_\omega)$
where $U_\omega=\bigcup_n U_n$, $\ell_\omega=\bigcup_n \ell_n$,
$E_\omega=\bigcup_n E_n$ and $W_\omega=\bigcup_n W_n$.

By the end of the construction we will achieve the following 
\emph{saturation} condition:
\begin{description}
\item[Sat]
for every $(u,v)\in E_\omega$ and $\tau\comp\sigma\in\ell_\omega(u,v)$, we have 
$\tau\in\ell_\omega(u,w)$ and $\sigma\in\ell_\omega(w,v)$ for some $w\in U_\omega$.
\end{description}

In the 0th step of the step-by-step construction we define $G_0$
by creating a witness edge for our fixed term $\theta$.
We choose $u_0,v_0\in\omega$ such that $u_0=v_0$ iff $\theta\le\ide$ is derivable.
We define
\begin{align*}
\ell_0(u_0,v_0)&=\mathcal{F}(\mathcal{E}\comp\theta\comp\mathcal{E})=\mathcal{F}(\theta)\\
\ell_0(u_0,u_0)&=\ell_0(v_0,v_0)=\mathcal{E}
\end{align*} 
and we label $(v_0,u_0)$ by $\emptyset$ in case $u_0\ne v_0$.
Observe that $\ell_0$ is well defined.
Indeed, in case $\theta\le\ide$, we have that 
$\mathcal{E}=\theta^\uparrow=\mathcal{F}(\theta)$.
All non-empty edges constructed so far will be witness edges: $W_0=E_0$.

It is easy to see that the inductive conditions are true.
For Comp we note that,
for every $\epsilon,\epsilon'\in\ell_0(u_0,u_0)=\ell_0(v_0,v_0)=\mathcal{E}$, we have
$\epsilon\comp\epsilon'=\epsilon\meet\epsilon'\in\mathcal{E}$ 
by \eqref{eq:fun},
and 
$\epsilon\comp\theta\comp\epsilon'=\theta\in\ell_0(u_0,v_0)$
by
the definition of $\mathcal{E}$.
DR easily follows from \eqref{eq:emon}.
Note that $\theta'\notin\ell_0(u_0,v_0)$ whenever
$\mbox{\rm Ax}^i(\meet,\comp,0,\ide)\not\vdash\theta\le\theta'$,
since $\theta'\notin\mathcal{F}(\theta)=\theta^\uparrow$.

In the $(n+1)$th step we assume inductively that 
$G_n=(U_n,\ell_n,E_n,W_n)$ with $U_n\subset \omega$ and 
$W_n\subseteq E_n\subseteq U_n\times U_n$
has been constructed and that $G_n$ satisfies the inductive conditions.
We also assume that there is a fair scheduling function 
$\Sigma\colon\omega\to\omega\times\omega\times F_X$,
i.e., every element of $\omega\times\omega\times F_X$
appears infinitely often in the range of $\Sigma$.

Assume that $\Sigma(n+1)=(u,v,\tau\comp\sigma)$,
$(u,v)\in E_n$ and $\tau\comp\sigma\in\ell_n(u,v)$,
otherwise we define $G_{n+1}= G_n$.
Recall that $\ell_n(u,u)=\ell_n(v,v)=\mathcal{E}$ by DR.
Let $\rho$ be such that 
$\ell_n(u,v)=\mathcal{F}(\mathcal{E}\comp\rho\comp\mathcal{E})$, 
and $\rho'$ be such that
$\ell_n(v,u)=\mathcal{F}(\mathcal{E}\comp\rho'\comp\mathcal{E})$ 
in case $(v,u)\in E_n$ as well.
Thus $\epsilon_0\comp\rho\comp\epsilon_0\le\tau\comp\sigma$ 
for some subidentity $\epsilon_0\in\mathcal{E}$.
Observe that, for any subidentity $\epsilon$,
\begin{equation}
(\epsilon_0\meet\epsilon)\comp\rho\comp(\epsilon_0\meet\epsilon)\meet\tau\comp\sigma=
(\epsilon_0\meet\epsilon)\comp\rho\comp(\epsilon_0\meet\epsilon)\meet
(\epsilon_0\meet\epsilon)\comp\tau\comp(\epsilon_0\meet\epsilon)\comp\sigma
\label{eq:eps}
\end{equation}
by using $\epsilon_0\meet\epsilon=(\epsilon_0\meet\epsilon)\comp(\epsilon_0\meet\epsilon)$,
\eqref{eq:int2} and~\eqref{eq:con2}.

We will assume that $\ide\notin\mathcal{F}(\mathcal{E}\comp\tau\comp\mathcal{E})$ 
and $\ide\notin\mathcal{F}(\mathcal{E}\comp\sigma\comp\mathcal{E})$,
since we would have the required edges otherwise. 
Indeed, we have
\begin{equation}\label{eq:nonide}
\ide\in\mathcal{F}(\mathcal{E}\comp\tau\comp\mathcal{E})\mbox{ implies }
\tau\in\ell_n(u,u)\mbox{ and }\sigma\in\ell_n(u,v)
\end{equation}
because of the following.
Assume that $\epsilon'\comp\tau\comp\epsilon'\le\ide$ for some
subidentity $\epsilon'\in\mathcal{E}$.
Let $\epsilon=\epsilon_0\meet\epsilon'$.
Then $\epsilon\comp\rho\comp\epsilon\in\ell_\omega(u,v)$ and
\begin{align*}
\epsilon\comp\rho\comp\epsilon&=(\epsilon\comp\rho\comp\epsilon)\meet(\tau\comp\sigma)\\
&=(\epsilon\comp\rho\comp\epsilon)\meet
(\epsilon\comp\tau\comp\epsilon\comp\sigma)
&&\mbox{by }\eqref{eq:eps}\\
&\le\rho\meet\sigma
&&\mbox{by }\epsilon\comp\tau\comp\epsilon\le\ide\\
&\le\sigma&&
\end{align*}
whence $\sigma\in\ell_n(u,v)$.
Next we claim that $\epsilon\comp\tau\comp\epsilon\in\ell_n(u,u)$.
Indeed, by using~\eqref{eq:eps},~\eqref{eq:con2}
we get
\begin{align*}
(\epsilon\comp\tau\comp\epsilon)\comp(\epsilon\comp\rho\comp\epsilon)&=
(\epsilon\comp\tau\comp\epsilon)\comp(\tau\comp\sigma\meet
\epsilon\comp\rho\comp\epsilon)\\
&=(\epsilon\comp\tau\comp\epsilon)\comp
((\epsilon\comp\tau\comp\epsilon\comp\sigma)\meet
\epsilon\comp\rho\comp\epsilon)\\
&=(\epsilon\comp\tau\comp\epsilon\comp\sigma)
\meet(\epsilon\comp\rho\comp\epsilon)\\
&=(\tau\comp\sigma)
\meet(\epsilon\comp\rho\comp\epsilon)\\
&=\epsilon\comp\rho\comp\epsilon
\end{align*}
whence $\tau\ge\epsilon\comp\tau\comp\epsilon\in\ell_n(u,u)$ by DR.
The case $\ide\in\mathcal{F}(\mathcal{E}\comp\sigma\comp\mathcal{E})$
is treated similarly.

Take a point $w\in\omega\smallsetminus U_n$ and define
\begin{align*}
\ell_{n+1}(w,w)&=\mathcal{E}\\
\ell_{n +1}(t,w)&=\mathcal{F}(\ell_n(t,u)\comp\tau\comp\mathcal{E})\\
\ell_{n +1}(w,s)&=\mathcal{F}(\mathcal{E}\comp\sigma\comp\ell_n(v,s))
\end{align*}
whenever $(t,u),(v,s)\in E_n$.
In particular,
$$
\ell_{n +1}(u,w)=
\mathcal{F}(\mathcal{E}\comp\tau\comp\mathcal{E})
\quad\mbox{and}\quad
\ell_{n +1}(w,v)=
\mathcal{F}(\mathcal{E}\comp\sigma\comp\mathcal{E})
$$
since $\ell_n(u,u)=\ell_n(v,v)=\mathcal{E}$.
By Fun and DR there are $\xi$ and $\chi$ such that 
$\ell_n(t,u)=\mathcal{F}(\mathcal{E}\comp\xi\comp\mathcal{E})$
and
$\ell_n(v,s)=\mathcal{F}(\mathcal{E}\comp\chi\comp\mathcal{E})$.
Then using~\eqref{eq:int2}
\begin{align*}
\ell_{n +1}(t,w)&=\mathcal{F}(\mathcal{E}\comp\xi\comp\mathcal{E}\comp\tau\comp\mathcal{E})
=\mathcal{F}(\mathcal{E}\comp\xi\comp\tau\comp\mathcal{E})\\
\ell_{n +1}(w,s)&=\mathcal{F}(\mathcal{E}\comp\sigma\comp\mathcal{E}\comp\chi\comp\mathcal{E})
=\mathcal{F}(\mathcal{E}\comp\sigma\comp\chi\comp\mathcal{E})
\intertext{and similarly}
\ell_{n +1}(w,u)&=
\mathcal{F}(\mathcal{E}\comp\sigma\comp\rho'\comp\mathcal{E})\\
\ell_{n +1}(v,w)&=
\mathcal{F}(\mathcal{E}\comp\rho'\comp\tau\comp\mathcal{E})
\end{align*}
when $(v,u)\in E_n$.
Thus the labels are fundamental filters.
For all other edges, we let 
$\ell_{n+1}(x,y)=\ell_{n}(x,y)$ if  $(x,y)\in E_n$,
and $\ell_{n+1}(x,y)=\emptyset$ otherwise. 
The witness edges are those of $G_n$ (i.e., $W_n$)
and $(u,w)$, $(w,v)$ and $(w,w)$.
See Figure~\ref{fig:comp3}, where we show typical elements of the labels.
\begin{figure}[h]
\[
\xymatrix{
&&&&w\ar@(ul,ur)^{\epsilon}
\ar@/_1.5pc/[dddllll]_{\epsilon\comp\sigma\comp\rho'\comp\epsilon}
\ar[dddrrrr]_(0.6){\epsilon\comp\sigma\comp\epsilon}
\ar[ddddddr]_(.9){\epsilon\comp\sigma\comp\chi\comp\epsilon}&&&&\\
&&&&&&&&\\
&&&&&&&&\\
u\ar@(ul,dl)_{\epsilon}
\ar[uururrr]_(0.4){\epsilon\comp\tau\comp\epsilon}
\ar@/^1pc/[rrrrrrrr]_{\epsilon\comp\rho\comp\epsilon}&&&&&&&&
v\ar@(ur,dr)^{\epsilon}\ar[dddlll]^{\epsilon\comp\chi\comp\epsilon}
\ar@/_1.5pc/[uuullll]_{\epsilon\comp\rho'\comp\tau\comp\epsilon}
\ar@/^1pc/[llllllll]^{\epsilon\comp\rho'\comp\epsilon}
\\
&&&&&&&&\\
&&&&&&&&\\
&&&t\ar@(ul,dl)_{\epsilon}\ar[uuulll]^{\epsilon\comp\xi\comp\epsilon}
\ar[uuuuuur]_(.2){\epsilon\comp\xi\comp\tau\comp\epsilon}&&
s\ar@(ur,dr)^{\epsilon}&&&
}
\]
\caption{Successor step}\label{fig:comp3}
\end{figure}
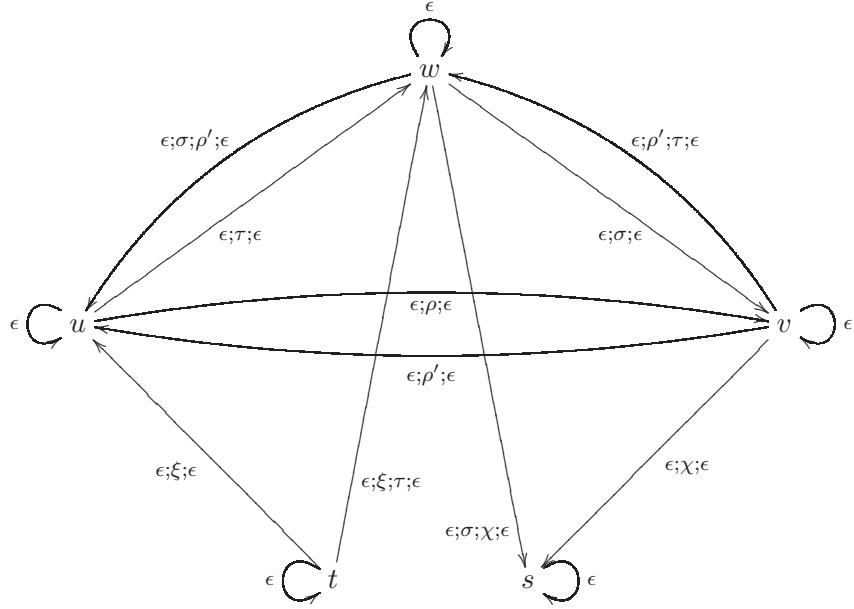

Observe that $E_{n+1}$ is reflexive, transitive and generated by $W_{n+1}$
(e.g., in case $(v,u)\in E_n$, the edge $(w,u)$ can be ``decomposed'' to $(w,v)\in W_{n+1}$
and $(v,u)$ which is generated by $W_n$ by the inductive condition).
Ide holds for witness edges, since the new irreflexive witness edges 
$(u,w)$ and $(w,v)$ avoid $\ide$ by the assumption on $\tau$ and $\sigma$.

Next we check that Comp is maintained as well.
Both $\ell_{n+1}(w,w)\comp\ell_{n+1}(w,s)\subseteq\ell_{n+1}(w,s)$ and
$\ell_{n+1}(t,w)\comp\ell_{n+1}(w,w)\subseteq\ell_{n+1}(t,w)$
easily follow from the definition of the labels.
Similarly we get
$\ell_{n+1}(w,s)\comp\ell_{n+1}(s,s)\subseteq\ell_{n+1}(w,s)$ and
$\ell_{n+1}(t,t)\comp\ell_{n+1}(t,w)\subseteq\ell_{n+1}(t,w)$
using that $G_n$ satisfies DR and Comp.
If $(w,t),(t,w)\in E_{n+1}$, we need 
$\ell_{n+1}(w,t)\comp\ell_{n+1}(t,w)\subseteq\ell_{n+1}(w,w)$
as well.
Let $\chi\in\ell_{n+1}(v,t)$ and $\xi\in\ell_{n+1}(t,u)$ 
so that $\epsilon\comp\sigma\comp\chi\comp\epsilon\in\ell_{n+1}(w,t)$
and $\epsilon\comp\xi\comp\tau\comp\epsilon\in\ell_{n+1}(t,w)$
for some subidentity $\epsilon\in\mathcal{E}$.
Then
\begin{align*}
\epsilon\comp\sigma\comp\chi\comp\epsilon\comp\xi\comp\tau\comp\epsilon\meet\epsilon&
\ge\epsilon\comp\sigma\comp\rho'\comp\tau\comp\epsilon\meet\epsilon\\
&=\epsilon\comp\sigma\comp\tau\comp\rho'\comp\epsilon\meet\epsilon\\
&\ge\epsilon\comp\rho\comp\rho'\comp\epsilon\meet\epsilon\\
&\in\ell_n(u,u)=\mathcal{E}
\end{align*}
by \eqref{eq:int1} and Comp for $G_n$.
Thus 
$\epsilon\comp\sigma\comp\chi\comp\epsilon\comp\xi\comp\tau\comp\epsilon
\in\mathcal{E}=\ell_{n+1}(w,w)$
as desired.
The proof of
$\ell_{n+1}(t,w)\comp\ell_{n+1}(w,t)\subseteq\ell_{n+1}(t,t)$ is similar.
Next we show
$\ell_{n+1}(t,w)\comp\ell_{n+1}(w,s)\subseteq\ell_{n+1}(t,s)$.
Indeed,  we have
\begin{align*}
\mathcal{F}(\ell_{n}(t,u)\comp\tau\comp\mathcal{E})\comp
\mathcal{F}(\mathcal{E}\comp\sigma\comp\ell_{n}(v,s))
&\subseteq
\mathcal{F}(\ell_{n}(t,u)\comp\tau\comp\sigma\comp\ell_{n}(v,s))\\
&\subseteq
\mathcal{F}(\ell_{n}(t,u)\comp\rho\comp\ell_{n}(v,s))\\
&\subseteq
\ell_{n}(t,s)=\ell_{n+1}(t,s)
\end{align*}
by Comp for $G_n$.
For
$\ell_{n+1}(s,t)\comp\ell_{n+1}(t,w)\subseteq\ell_{n+1}(s,w)$
we have
\begin{align*}
\ell_n(s,t)\comp\mathcal{F}(\ell_n(t,u)\comp\tau\comp\mathcal{E})
&\subseteq
\mathcal{F}(\ell_{n}(s,t)\comp\ell_n(t,u)\comp\tau\comp\mathcal{E})\\
&\subseteq
\mathcal{F}(\ell_{n}(s,u)\comp\tau\comp\mathcal{E})\\
&=\ell_{n+1}(s,w)
\end{align*}
by Comp for $G_n$,
and similarly for
$\ell_{n+1}(w,s)\comp\ell_{n+1}(s,t)\subseteq\ell_{n+1}(w,t)$.

Finally we check DR.
First let $\xi\in\ell_{n+1}(t,u)$ and $\epsilon\in\mathcal{E}$ so that 
$\xi\comp\tau\comp\epsilon\in
\mathcal{F}(\ell_{n+1}(t,u)\comp\tau\comp\mathcal{E})=\ell_{n+1}(t,w)$,
and assume that $\epsilon'\le\ide$ such that
$\xi\comp\tau\comp\epsilon\comp\epsilon'=
\xi\comp\tau\comp\epsilon$.
We show that $\epsilon'\in\mathcal{E}=\ell_{n+1}(w,w)$.
We have
$\xi\comp\tau\comp\sigma\comp\epsilon\comp\epsilon'=
\xi\comp\tau\comp\epsilon\comp\epsilon'\comp\sigma=
\xi\comp\tau\comp\epsilon\comp\sigma=
\xi\comp\tau\comp\sigma\comp\epsilon\in\ell_{n}(t,v)$,
whence $\epsilon'\in\ell_n(v,v)=\mathcal{E}=\ell_{n+1}(w,w)$
by DR for $G_n$.
That is, $\mathcal{E}(\xi\comp\tau\comp\epsilon)\subseteq\ell_{n+1}(w,w)$.
A similar proof shows DR for $\ell_{n+1}(w,s)$.

This finishes the successor step yielding
$G_{n+1}=(U_{n+1},\ell_{n+1},E_{n+1},W_{n+1})$.
After the construction terminates we end up with a labelled structure
$G_\omega=(U_\omega,\ell_\omega, E_\omega,W_\omega)$ 
such that $U_\omega=\bigcup_n U_n$, $\ell_\omega=\bigcup_n\ell_n$,
$E_\omega=\bigcup_n E_n$ and $W_\omega=\bigcup_n W_n\subseteq E_\omega$.
Observe that $G_{\omega}$ satisfies Sat, 
since the fair scheduling function $\Sigma$ ensures that every possible 
composition ``defect'' has been taken care of.
But we have not showed yet that Ide is satisfied in general,
we will do this shortly.
$G_\omega$ satisfies the other inductive conditions
(since so does every $G_n$).

\subsection{Graphs and algebras}

We define a valuation $\iota$ of variables on $U_\omega\times U_\omega$:
\begin{equation}\label{eq:val}
\iota(x)=\{(u,v)\in E_\omega: x\in\ell_\omega(u,v)\}
\end{equation}
for every variable $x\in X$.
Let $\mathfrak{A}_\theta$ be the subalgebra of
$(\wp(U_\omega\times U_\omega),\meet,\comp,0,\ide)$ 
generated by $\{\iota(x):x\in X\}$.
Next we show that $\mathfrak{A}_\theta\in\R^i(\meet,\comp,0,\ide)$.

Let $\mathfrak{A}_\theta(u,v)$ denote the set of those elements (a filter) 
that hold at $(u,v)$.

\begin{lemma}\label{lem:gen}
For every $u\in U_\omega$ and term $\tau$,
\begin{equation}\label{eq:refl}
\tau\in\ell_\omega(u,u)\text{ implies }\tau\in\mathfrak{A}_\theta(u,u).
\end{equation}
\end{lemma}

\begin{proof}
We will prove the lemma together with
\begin{equation}\label{eq:nrefl}
\tau\not\le\ide\text{ and }\tau\in\ell_\omega(u,v)\text{ imply }
\tau\in\mathfrak{A}_\theta(u,v)
\end{equation}
for every $(u,v)\in U_\omega\times U_\omega$.
We will use simultaneous induction on the complexity of terms.

The base case, when $\tau$ is a variable, the identity constant $\ide$, or $0$ is straightforward
by the definition of the valuation $\iota$ and the fact that the labels are proper filters.

Next assume that $\tau=\sigma\meet\rho$.
Recall that $\ell_\omega(u,v)$ is a filter, 
whence $\tau\in\ell_\omega(u,v)$ implies $\sigma,\rho\in\ell_\omega(u,v)$.
First consider the statement~\eqref{eq:nrefl}.
If $\tau\not\le\ide$, then also $\sigma,\rho\not\le\ide$,
whence we can apply the induction hypothesis (IH) for \eqref{eq:nrefl}, yielding 
$\sigma,\rho\in\mathfrak{A}_\theta(u,v)$ and
$\sigma\meet\rho\in\mathfrak{A}_\theta(u,v)$, since  $\mathfrak{A}_\theta(u,v)$ is a filter.
The proof for~\eqref{eq:refl} is analogous.

Finally assume that $\tau=\sigma\comp\rho$.
We start with showing~\eqref{eq:nrefl}.
First assume that $\sigma\le\ide$, whence $\rho\not\le\ide$ (by $\tau\not\le\ide$).
Then using~\eqref{eq:nonide} we get that
$\sigma\in\ell_\omega(u,u)$ and $\rho\in\ell_\omega(u,v)$.
By the IH for~\eqref{eq:refl} and~\eqref{eq:nrefl}
we have $\sigma\in\mathfrak{A}_\theta(u,u)$ and $\rho\in\mathfrak{A}_\theta(u,v)$.
Thus $\sigma\comp\rho\in\mathfrak{A}_\theta(u,v)$.
The case $\rho\le\ide$ is treated similarly.
The last case is when $\sigma,\rho\not\le\ide$.
By Sat we have $w$ such that $\sigma\in\ell_\omega(u,w)$ and $\rho\in\ell_\omega(w,v)$.
Using the IH for~\eqref{eq:nrefl} we get
$\sigma\in\mathfrak{A}_\theta(u,w)$ and $\rho\in\mathfrak{A}_\theta(w,v)$,
whence  $\sigma\comp\rho\in\mathfrak{A}_\theta(u,v)$.

For~\eqref{eq:refl} we argue as follows.
First assume that $\sigma\le\ide$.
Using~\eqref{eq:nonide} we get that
$\sigma,\rho\in\ell_\omega(u,u)$.
By the IH for~\eqref{eq:refl} we get
$\sigma,\rho\in\mathfrak{A}_\theta(u,u)$,
whence  $\sigma\comp\rho\in\mathfrak{A}_\theta(u,u)$.
The case for $\rho\le\ide$ is similar.
Finally, the case $\sigma,\rho\not\le\ide$ is treated precisely like for~\eqref{eq:nrefl}
at the end of the previous paragraph.
\end{proof}

\begin{lemma}\label{lem:w1}
For every term $\tau$ and for every edge $(u,v)\in U_\omega\times U_\omega$,
\begin{equation*}
\tau\in\mathfrak{A}_\theta(u,v)\mbox{ implies }\tau\in\ell_\omega(u,v).
\end{equation*}
\end{lemma}

\begin{proof}
This is an easy induction on the complexity of terms 
(using Comp for composition).
\end{proof}

\begin{lemma}\label{lem:int}
$\mathfrak{A}_\theta\in \R^i(\meet,\comp,0,\ide)$.
\end{lemma}

\begin{proof}
Clearly, $\mathfrak{A}_\theta$ is an algebra of relations.
Recall that we have $\ell_\omega(u,u)=\mathcal{E}(\theta)=\mathcal{E}$ for every $u\in U_\omega$.
Then by Lemma~\ref{lem:gen} and Lemma~\ref{lem:w1} we get
\begin{equation}\label{eq:uni}
\mathfrak{A}_\theta(u,u)=\mathfrak{A}_\theta(v,v)\text{, for every }u,v\in U_\omega.
\end{equation}
Recall that $\mathfrak{A}_\theta$ is generated by $\{\iota(x): x\in X\}$,
thus every element of $\mathfrak{A}_\theta$ is the interpretation of a term.
Now assume indirectly that the interpretation of the identity constant $\ide$ 
in $\mathfrak{A}_\theta$ is not an atom, i.e., there is a term $\tau$ such that
its interpretation is a proper, nonempty subset $T$ of $\{(u,u):u\in U_\omega\}$.
Then $\tau\in\mathfrak{A}_\theta(u,u)$ iff $(u,u)\in T$, contradicting to 
\eqref{eq:uni}.
Thus $\mathfrak{A}_\theta$ is an integral algebra.
\end{proof}

Next we show that $\mathfrak{A}_\theta$ is a witness for the non-derivable equations
of the form $\theta\le\theta'$ for our fixed $\theta$.

\begin{lemma}\label{lem:w}
For every term $\tau$ and for every edge $(x,y)\in U_\omega\times U_\omega$,
\begin{equation*}
\tau\in\ell_\omega(x,y)\mbox{ implies }\tau\in\mathfrak{A}_\theta(x,y).
\end{equation*}
\end{lemma}

\begin{proof}
We already showed the lemma for the case $x=y$ and for the general case 
with the restriction that $\tau\not\le\ide$, see the proof of Lemma~\ref{lem:gen} above. 
Then it would suffice to show that $\ide$ cannot occur in the label of the edge $(x,y)$
whenever $x\ne y$.
Indeed, our proof will show this (see~\eqref{eq:ind2} below),
but we will go through all the cases for the sake of clarity.

We use induction over the ``distance'' between $x$ and $y$.
To this end we define $\path(x,y)$ for $(x,y)\in E_\omega$ by recursion.
For $(x,y)\in E_0$, we let $\path(x,y)=(x,y)$.
Now assume that $(x,y)\in E_{n+1}\smallsetminus E_n$ and
we already defined $\path$ for the elements of $E_n$.
Assume that $U_{n+1}=U_n\cup \{w\}$ and that we expanded $U_n$ by $w$
because of some $\sigma\comp\rho\in\ell_n(u,v)$.
Again we let $\path(x,y)=(x,y)$ for $x,y\in\{u,w,v\}$.
Now assume that $t\ne u$, $(t,u)\in E_n$ and $\path(t,u)=(t,z_0,\dots,z_k,u)$.
Then we let 
\begin{equation}\label{eq:path}
\path(t,w)=(t,z_0,\dots,z_k,u,w)
\end{equation}
and define $\path(w,s)$ for $(v,s)\in E_n$ analogously.
Finally, $d(x,y)$ is defined as the length of $\path(x,y)$.

The base case is when $(x,y)\in W_\omega$ is a witness edge: $d(x,y)=1$.
The base case will be established by induction on terms.
The case of a variable is straightforward by the definition~\eqref{eq:val} 
of the valuation $\iota$.
The case of $0$ follows from the fact that we used proper filters as labels.
Next assume that $\tau$ is the constant $\ide$.
Observe that $\ide\in\ell_\omega(x,y)$ implies $x=y$, 
since irreflexive witness edges avoid $\ide$. 
Then $\ide\in\mathfrak{A}_\theta(x,x)$, since $\mathfrak{A}_\theta$ is representable.
The case for $\meet$ follows from the fact that we used filters as labels.
For the case of $\comp$ assume that $\tau$ is $\sigma\comp\rho$ and
that $\sigma\comp\rho\in\ell_\omega(x,y)$.
Recall that 
we had either $\sigma\in\ell_\omega(x,x)$ and $\rho\in\ell_\omega(x,y)$,
or $\sigma\in\ell_\omega(x,y)$ and $\rho\in\ell_\omega(y,y)$,
or we constructed witness edges
$(x,z)$ and $(z,y)$ such that 
$\sigma\in\ell_\omega(x,z)$ and $\rho\in\ell_\omega(z,y)$.
Thus $\sigma\in\mathfrak{A}_\theta(x,z)$ and $\rho\in\mathfrak{A}_\theta(z,y)$ by the IH,
whence $\sigma\comp\rho\in\mathfrak{A}_\theta(x,y)$ as desired.

For the inductive case assume that $d(x,y)=k+1>1$.
Again we use induction on terms.
The cases for variables and $0$ are as in the base step $(x,y)\in W_\omega$.
For the case of $\ide$ we show that $G_\omega$ in fact satisfies Ide:
\begin{equation}\label{eq:ind2}
x\ne y\mbox{ implies }\ide\notin\ell_\omega(x,y).
\end{equation}
Assume that $(x,y)$ was created in the $(n+1)$th step of the construction.
Recall that during the construction we defined $\ell_{n+1}(x,y)$
as $\mathcal{F}(\ell_{n+1}(x,z)\comp\ell_{n+1}(z,y))$ for some $z$ such that
either $(x,z)\in E_n$ and $(z,y)\in W_{n+1}$, 
or $(x,z)\in W_{n+1}$ and $(z,y)\in E_n$.
Wlog assume the former.
Then, using the notation in Figure~\ref{fig:comp3}, we have $x=t$, $y=w$ and $z=u$.
Recall that $\path(t,w)$ is defined by adding the step $(u,w)$ at
the end of $\path(t,u)$, see~\eqref{eq:path}. 
Hence $d(x,z)=d(t,u) < d(t,w)=d(x,y)$.

Let $\ell_{n+1}(x,z)=\mathcal{F}(\mathcal{E}\comp\rho_1\comp\mathcal{E})$ 
and $\ell_{n+1}(z,y)=\mathcal{F}(\mathcal{E}\comp\rho_2\comp\mathcal{E})$,
whence 
$\ell_{n+1}(x,y)=\mathcal{F}(\mathcal{E}\comp\rho_1\comp\rho_2\comp\mathcal{E})$ 
by the construction.
Since $d(x,z)<d(x,y)$ and $(z,y)\in W_{n+1}$,
we can apply the induction hypothesis:
$\epsilon\comp\rho_1\comp\epsilon\in\mathfrak{A}_\theta(x,z)$ and
$\epsilon\comp\rho_2\comp\epsilon\in\mathfrak{A}_\theta(z,y)$
for every $\epsilon\in\mathcal{E}$.
Thus we get
$\epsilon\comp\rho_1\comp\rho_2\comp\epsilon\in\mathfrak{A}_\theta(x,y)\not\ni\ide$.
Assume, for a contradiction, that $\ide\in\ell_{n+1}(x,y)$.
Then $\epsilon\comp\rho_1\comp\rho_2\comp\epsilon\le\ide$ is derivable from
$\mbox{\rm Ax}^i(\meet,\comp,0,\ide)$ for some $\epsilon\in\mathcal{E}$.
But $\epsilon\comp\rho_1\comp\rho_2\comp\epsilon\le\ide$ is not valid in 
$\R^i(\meet,\comp,0,\ide)$ as witnessed by $\mathfrak{A}_\theta$, a contradiction.
Hence~\eqref{eq:ind2} holds.
Thus we have that $\ide\in\ell_\omega(x,y)$ only if $x=y$, 
and then $\ide\in\mathfrak{A}_\theta(x,y)$ as required.

The case of $\meet$ follows as above.
Finally assume that $\tau=\sigma\comp\rho$ and $\sigma\comp\rho\in\ell_{n}(x,y)$.
We have to consider three cases.
The first is when $\sigma\in\ell_{n}(x,x)$ and $\rho\in\ell_{n}(x,y)$.
Then $\sigma\in\mathfrak{A}_\theta(x,x)$ (since $(x,x)\in W_\omega$)
and $\rho\in\mathfrak{A}_\theta(x,y)$ (since $\rho$ is a simpler term than $\sigma\comp\rho$).
Thus $\sigma\comp\rho\in\mathfrak{A}_\theta(x,y)$.
The case $\sigma\in\ell_{n}(x,y)$ and $\rho\in\ell_{n}(y,y)$ is similar.
Finally, if neither of the above cases apply, then we created witness edges
$(x,z)$ and $(z,y)$ such that 
$\sigma\in\ell_{\omega}(x,z)$ and $\rho\in\ell_{\omega}(z,y)$.
By the base case we get
$\sigma\in\mathfrak{A}_\theta(x,z)$ and $\rho\in\mathfrak{A}_\theta(z,y)$,
whence $\sigma\comp\rho\in\mathfrak{A}_\theta(x,y)$ as desired.
\end{proof}

\begin{lemma}\label{lem:wit}
For every $\theta'$ such that 
$\mbox{\rm Ax}^i(\meet,\comp,0,\ide)\not\vdash \theta\le\theta'$,
we have $\R^i(\meet,\comp,0,\ide)\not\models\theta\le\theta'$.
\end{lemma}

\begin{proof}
By Lemma~\ref{lem:int} we have $\mathfrak{A}_\theta\in \R^i(\meet,\comp,0,\ide)$.
By the initial step of the step-by-step construction we have
$\theta\in\ell_\omega(u_0,v_0)=\mathcal{F}(\theta)$ and $\theta'\notin\ell_\omega(u_0,v_0)$.
So by Lemma~\ref{lem:w} we get $\theta\in\mathfrak{A}_\theta(u_0,v_0)$ and
by Lemma~\ref{lem:w1} we have $\theta'\notin\mathfrak{A}_\theta(u_0,v_0)$.
\end{proof}

\begin{remark}[Representing the free algebra]
In the above construction we fixed a term $\theta$ 
and constructed an algebra $\mathfrak{A}_\theta\in\R^i(\meet,\comp,0,\ide)$.
We can repeat the same construction for every non-zero element $\theta$ of the
free algebra $\mathfrak{F}_X$, resulting in 
$\mathfrak{A}_\theta\in\R^i(\meet,\comp,0,\ide)$.
It is not difficult to show that 
$\mathfrak{F}_X$ can be embedded into 
$\prod_{\theta\ne 0}\mathfrak{A}_\theta\in\V(\R^i(\meet,\comp,0,\ide))$.
\end{remark}

\section{Conclusions}

The varieties $\V(\R(\Lambda))$ generated by algebras of binary relations 
of the similarity types
$\Lambda=(\meet,\comp,\ide)$ and $\Lambda=(\join,\meet,\comp,\ide)$ 
were stated to be finitely axiomatizable in \cite{AM-axi-11} 
(Theorem~4.3 and Theorem~4.1(1), respectively), 
but their proofs relied on false lemmas. 
See \cite{AM-err-14}.

In more detail, the third case of Definition~4.6 is ambiguous, 
and Lemmas~4.7 and~4.8 are not true. 
These lemmas are used in the proof of Theorem~4.3. 
As a consequence, the proof of Theorem~4.3 breaks down for the
equation $\ide\meet x\comp y\le x\comp(\ide\meet y\comp x)\comp y$.
This equation is easily seen to be valid, but so far we did not manage
to derive it from the axioms of Theorem~4.3.
In fact, we conjecture that this equation does not follow from the axioms
presented in \cite{AM-axi-11}.
Theorem~4.3 is used in the proof of Theorem~4.1(1). 

Since the main results of this paper, Theorem~\ref{thm:main} and~\ref{thm:dl},
give only a solution for the special case of integral algebras, we state the following as an open problem.

\begin{problem}\label{prob}
Are the varieties generated by algebras of binary relations 
of the similarity types $(\meet,\comp,\ide)$ and $(\join,\meet,\comp,\ide)$ 
finitely axiomatizable?
\end{problem}

In \cite{AMN-equ-11}, we stated that $\V(\Lang(\join,\meet,\comp 0, \ide))$
is finitely axiomatizable by a certain set of axioms, Corollary~3.7.
The proof was based on the theorems of \cite{AM-axi-11} mentioned above,
hence it is not correct.
Luckily Theorem~\ref{thm:lang} above provides a satisfactory solution in this case.

Finally we mention a corollary related to commutative algebras.
Let $\R^c(\Lambda)$ denote the class of relation algebras of signature $\Lambda$
that satisfy the commutativity axiom
\begin{equation}\label{eq:comm}
x\comp y = y\comp x
\end{equation}
for every $x$ and $y$.
It is easy to see that $\V(\R^c(\meet,\comp, 0, \ide))$ satisfies 
$\mbox{\rm Ax}^i(\meet,\comp,0,\ide)$.
Thus our construction can be applied in this case as well.

\begin{corollary}
The varieties $\V(\R^c(\meet,\comp, 0, \ide))$, $\V(\R^c(\join,\meet,\comp, 0, \ide))$
are finitely axiomatized by commutativity~\eqref{eq:comm} and
$\mbox{\rm Ax}(\meet,\comp,0,\ide)$, $\mbox{\rm Ax}(\join,\meet,\comp,0,\ide)$, respectively.
\end{corollary}

\paragraph{Acknowledgements}
The author is grateful to Hajnal Andr\'eka, Robin Hirsch and Istv\'an N\'emeti
for helpful comments.


\begin{thebibliography}{MMMMM}

\bibitem[AB95]{AB-equ-95}
{H. Andr\'eka and D.A. Bredikhin},
``The equational theory of union-free algebras of relations'',
{\it Algebra Universalis},
33:516--532, 1995.

\bibitem[AM11]{AM-axi-11}
{H. Andr\'eka and Sz.\ Mikul\'as},
``Axiomatizability of positive algebras of binary relations'',
{\it Algebra Universalis},
66:7--34, 2011.

\bibitem[AM14]{AM-err-14}
{H. Andr\'eka and Sz.\ Mikul\'as},
``Errata to ``Axiomatizability of positive algebras of binary
relations'', AU 66:7--34'',
available at 
\verb+www.dcs.bbk.ac.uk/~szabolcs/AM-AU-err6.pdf+


\bibitem[AMN11]{AMN-equ-11}
{H. Andr\'eka, Sz.\ Mikul\'as and I. N\'emeti},
``The equational theory of {K}leene lattices'',
{\it Theoretical Computer Science},
412:7099--7108, 2011.


\bibitem[Br93]{bre:93}
{D.A. Bredikhin},
``The equational theory of relation algebras with positive operations'',
(in Russian), {\it Izv.\ Vyash.\ Uchebn.\ Zaved.\ Math.},
3:23--30, 1993.

\bibitem[HM07]{HM-rep-07} 
{R. Hirsch and Sz. Mikul\'as},  
``Representable semilattice-ordered monoids'',  
\emph{Algebra Universalis}, 57:333--370, 2007.

\end{thebibliography}
\end{document}